\documentclass[letterpaper,10pt,conference]{ieeeconf}

\IEEEoverridecommandlockouts
\usepackage{graphics}
\usepackage{epsfig}
\usepackage{mathptmx}
\usepackage{times}
\usepackage{amsmath}
\usepackage{amssymb}
\usepackage[T1]{fontenc}
\usepackage{lmodern}
\usepackage{mathtools}
\usepackage{bm}
\usepackage{xcolor}

\newtheorem{theorem}{Theorem}
\newtheorem{lemma}{Lemma}
\newtheorem{proposition}{Proposition}

\newtheorem{definition}{Definition}

\newcommand{\mass}{\operatorname{mass}}
\newcommand{\spt}{\operatorname{spt}}
\newcommand{\defeq}{:=}

\title{\LARGE \bf
On the Peak Estimation of Scalar Polynomial Systems via Integrability
}

\author{Karol\'ina Sehnalov\'a$^{1}$%
\thanks{This work was funded by the European Union/M\v{S}MT \v{C}R under the
ROBOPROX project (reg.~no.~CZ.02.01.01/00/22 008/0004590) and by the Grant
Agency of the Czech Technical University in Prague (grant
No.~SGS25/145/OHK3/3T/13).}%
\thanks{$^{1}$Faculty of Electrical Engineering, Czech Technical University
in Prague, Technick\'a 2, CZ-16627 Prague, Czechia.
{\tt\small sehnakar@fel.cvut.cz}}%
}
\begin{document}

\maketitle
\thispagestyle{empty}
\pagestyle{empty}

%%%%%%%%%%%%%%%%%%%%%%%%%%%%%%%%%%%%%%%%%%%%%%%%%%%%%%%%%%%%%%%%%%%%%%%%%%%%%%%%
\begin{abstract}
We develop a convergent occupation measure framework for peak estimation of scalar polynomial dynamical systems with no a priori state bound. The key observation is that scalar trajectories satisfy a uniform weighted integrability bound whose exponent precisely compensates for the singularity introduced by polynomial compactification. This yields a finite measure formulation with polynomial data. We establish a superposition principle for the resulting weighted continuity equation, prove no relaxation gap and establish convergence of the associated moment-sum-of-squares hierarchy.
\end{abstract}

%%%%%%%%%%%%%%%%%%%%%%%%%%%%%%%%%%%%%%%%%%%%%%%%%%%%%%%%%%%%%%%%%%%%%%%%%%%%%%%%
\section{Introduction}
\label{sec:introduction}

Estimating extreme values attained along trajectories of dynamical systems is an important problem in nonlinear systems analysis, with applications in safety verification, reachability analysis, and worst-case performance assessment. In particular, estimating the maximum magnitude attained by trajectories originating from a prescribed set of initial conditions is relevant for verifying the behavior of closed-loop systems. A powerful approach to such problems is based on convex optimization, using either auxiliary functions \cite{fantuzzi_bounding_2020} or occupation measures
\cite{miller_peak_2024,miller2021,lasserre_nonlinear_2008}. The latter describe the evolution of ensembles of trajectories through positive measures, allowing nonlinear dynamical constraints to be expressed as linear equations and subsequently approximated using the moment-sum-of-squares (moment-SOS) hierarchy \cite{lasserre_positive_2010}.

A major difficulty arises when the state space is unbounded, as polynomial systems may exhibit trajectories escaping to infinity in finite time. Such behavior can be studied using Poincar\'e-type compactifications \cite{matsue_blow-up_2018,elias_critical_2006}, which transform the unbounded state space into a compact domain. In terms of polynomial optimization, i.e. without differential equations, this approach has been investigated in \cite{huang_homogenization_2022}. However, compactification introduces additional difficulties when combined with occupation measure methods. In particular, the transformation needed to recover polynomial dynamics may lead to measures with infinite mass, preventing the direct application of standard moment-SOS convergence results. A general framework for peak estimation over infinite time horizons and unbounded state spaces was recently developed in \cite{sehnalova2026infinite}. Nevertheless, convergence of the resulting numerical approximations has not been established in this general setting.

In this paper, we investigate the scalar case, which reveals a remarkable structural property that allows us to overcome this difficulty. Although scalar autonomous dynamics have a relatively simple qualitative structure and their finite time blow up behavior is well understood \cite{ishiwata_blowup_2022}, their compactification possesses a crucial advantage over the general multidimensional case. For polynomial dynamics of degree \(d\geq 3\), the standard compactification generally requires a time reparameterization involving the factor \(w^{-(d-1)}\), where \(w\) approaches zero as the state tends to infinity. This factor need not be integrable along trajectories exhibiting finite time blow up. In contrast, for scalar systems, an additional algebraic cancellation reduces the required factor to \(w^{-(d-2)}\). Remarkably, $d-2$ is precisely the largest integer exponent for which integrability is guaranteed along finite time escaping trajectories of scalar polynomial systems. Thus, a reduction by a single power makes the difference between a potentially infinite measure and one whose mass can be uniformly bounded.
\subsection{Contribution}
The main contribution of this paper is to identify and exploit this precise correspondence between the algebraic structure of the compactified dynamics and the integrability properties of scalar polynomial trajectories. We prove that the required integrability holds uniformly over all admissible trajectories, including those approaching finite time escape. This allows us to construct a finite mass occupation measure relaxation of the peak estimation problem on a compact semialgebraic domain. We further establish a superposition principle for the weighted continuity equation, representing feasible measures as superpositions of stopped trajectories, including those reaching infinity in finite time. This establishes the absence of a relaxation gap and allows us to
prove that the associated moment-SOS hierarchy yields a monotonically convergent sequence of upper bounds on the true peak value. These results identify a class of polynomial dynamical systems for which convergence can be rigorously guaranteed without assuming a priori bounds on the state variables or excluding finite time blow up. More broadly, they demonstrate that the possibility of finite time escape does not itself prevent convergence of moment-SOS methods; the crucial issue is whether the compactification can be made compatible with the integrability properties of the underlying dynamics.
\subsection{Outline}
The remainder of the paper is organized as follows. Section~\ref{sec:problem_statement} introduces the peak estimation problem and establishes the key integrability result. Section~\ref{sec:compactification} presents the compactification and occupation measure formulation. Section~\ref{sec:no_relaxation_gap} establishes the superposition principle and the absence of a relaxation gap. Section~\ref{sec:convergence} proves convergence of the moment-SOS hierarchy, followed by a numerical illustration.

\subsection{Notation}

We denote by $\mathcal M_+(\mathcal A)$ the set of nonnegative Radon measures supported on $\mathcal A$. The Dirac measure at a point $a$ is denoted by $\delta_a$. For a measure $\mu$, we denote its total mass by $\mass(\mu)$ and its support by $\spt(\mu)$. Given a measurable map $F:\mathcal A\to\mathcal B$ and a measure $\mu$ on $\mathcal A$, the pushforward measure $F_\#\mu$ is defined by
\(
(F_\#\mu)(\mathcal D)=\mu(F^{-1}(\mathcal D))
\)
for every Borel set $\mathcal D\subset\mathcal B$. Finally, $\mathbb R[x]$ denotes the ring of real polynomials in $x$.

\section{Problem statement}
\label{sec:problem_statement}

We consider the scalar autonomous polynomial dynamical system
\begin{equation}
\label{eq:system}
    \dot{x}(t)=f(x(t)),
\end{equation}
where \(x\in\mathbb R\) and \(f\in\mathbb R[x]\) is a polynomial of degree \(d\). Given an initial condition \(x_0\in\mathbb R\), we denote by \(x(\cdot;x_0)\) the corresponding maximal solution of \eqref{eq:system} and by \(t_{\max}(x_0)\in(0,\infty]\) its maximal existence time. Since the dynamics is polynomial, a solution may escape to infinity in finite time.

Let \(T>0\) be a fixed time horizon and let the set of initial conditions be the compact interval
\begin{equation}
\label{eq:initial_set}
    \mathcal X_0=[x_{\min},x_{\max}]\subset\mathbb R.
\end{equation}
We would like to estimate the largest magnitude reached by a trajectory originating in \(\mathcal X_0\) before the end of the time horizon or its maximal existence time. The problem is formulated as
\begin{align}
\label{problem_state_original}
    \bar p^*
    ={}&
    \sup_{x_0,\tau}
    |x(\tau;x_0)|
    \\
    \mathrm{s.t.}\quad
    &x_0\in\mathcal X_0,
    \qquad
    0\leq\tau\leq T,
    \qquad
    \tau<t_{\max}(x_0).
    \notag
\end{align}
Notice that \(\bar p^*\) may be finite or infinite.

To obtain a bounded objective, we therefore fix a positive even integer \(q\) and use the transformation \(
    h(x)
    :=
    \frac{x^q}{1+x^q},
\)
which is strictly increasing as a function of \(|x|\). We consider the equivalent problem
\begin{align}
\label{problem_state_bounded_cost}
    p^*
    ={}&
    \sup_{x_0,\tau}
    \frac{x(\tau;x_0)^q}
         {1+x(\tau;x_0)^q}
    \\
    \mathrm{s.t.}\quad
    &x_0\in\mathcal X_0,
    \qquad
    0\leq\tau\leq T,
    \qquad
    \tau<t_{\max}(x_0).
    \notag
\end{align}
Optimal values of \eqref{problem_state_original} and \eqref{problem_state_bounded_cost} are related by
\begin{equation}
\label{eq:weight_exponent}
    \bar p^*
    =
    \left(\frac{p^*}{1-p^*}\right)^{1/q}
\end{equation}
whenever \(p^*<1\). If \(p^*=1\), then \(\bar p^*=+\infty\). Thus, unbounded peak value is represented by the finite optimal value \(p^*=1\) in problem~\eqref{problem_state_bounded_cost}. For later use, we set \( \rho:=\max\{d-2,0\}. \)

For scalar polynomial systems of degree \(d\geq3\), the function \(|x(t)|^{d-2}\) remains integrable even when the solution escapes to infinity in finite time. We first establish a uniform bound of this type. If \(d\geq3\), choose \(R\geq1\) sufficiently large and \(c>0\) such that
\begin{equation}
\label{eq:polynomial_growth_bound}
    \mathcal X_0\subset[-R,R],
    \qquad
    |f(x)|\geq c|x|^d
    \quad\text{for }|x|\geq R,
\end{equation}
and define
\begin{equation}
\label{eq:explicit_weighted_mass_bound}
    C
    :=
    \begin{cases}
        T, & d\leq2,\\[1mm]
        T(1+R^q)^{(d-2)/q}
        +\dfrac{2^{(d-2)/q}}{cR}, & d\geq3.
    \end{cases}
\end{equation}

\begin{lemma}
\label{lem:uniform_integrability}
Given $\tau < t_{\max}(x_0)$, every trajectory in problem~\eqref{problem_state_bounded_cost} satisfies
\begin{equation}
\label{eq:uniform_integrability}
    \int_0^\tau
    \left(1+|x(t;x_0)|^q\right)^{\rho/q}dt
    \leq C
\end{equation}
\end{lemma}

\begin{proof}
If \(d\leq2\), then \(\rho=0\) and the integral is at most \(T\). Suppose that \(d\geq3\). The contribution accumulated while \(|x(t)|\leq R\) is at most \(T(1+R^q)^{(d-2)/q}\). By \eqref{eq:polynomial_growth_bound}, \(f\) has no zero outside \([-R,R]\). Hence, once a trajectory leaves this interval in an outward direction, \(r(t):=|x(t)|\) is increasing and \(\dot r(t)=|f(x(t))|\). Therefore,
\begin{align}
\int_0^\tau
    &\left(1+|x(t;x_0)|^q\right)^{\rho/q}dt \\&\leq
    T(1+R^q)^{(d-2)/q} + \int_R^\infty
    (1+r^q)^{(d-2)/q}\frac{dr}{cr^d}. \notag
\end{align}
Since \(R\geq1\), we have
\[
    (1+r^q)^{(d-2)/q}
    \leq 2^{(d-2)/q}r^{d-2}
    \qquad\text{for }r\geq R.
\]
Consequently,
\[
    \int_R^\infty
    (1+r^q)^{(d-2)/q}\frac{dr}{cr^d}
    \leq
    \frac{2^{(d-2)/q}}{c}
    \int_R^\infty\frac{dr}{r^2}
    =
    \frac{2^{(d-2)/q}}{cR},
\]
which proves \eqref{eq:uniform_integrability}.
\end{proof}

\section{Compactification and weighted occupation measures}
\label{sec:compactification}
\label{sec:occupation_measures}

We next compactify the state space in the same fashion as in \cite{sehnalova2025unbounded,sehnalova2026infinite}, using the transformation
\begin{equation}
\label{eq:compactification_map}
    z
    :=
    \frac{x}{(1+x^q)^{1/q}},
    \qquad
    w
    :=
    \frac{1}{(1+x^q)^{1/q}}.
\end{equation}
Since \(q\) is even, the transformed variables satisfy
\(
    z^q+w^q=1,
    w>0,
\)
and the inverse transformation is
\(
    x=\frac{z}{w}.
\)

The closure of the compactified state space is
\begin{equation}
\label{eq:compact_state_space}
    \mathcal B_q
    :=
    \left\{
        (z,w)\in\mathbb R^2
        \ : \
        z^q+w^q=1,\quad w\geq0
    \right\}.
\end{equation}
The two points
\(
    (1,0),
    (-1,0)
\)
represent \(+\infty\) and \(-\infty\), respectively. Moreover, the bounded objective becomes
\begin{equation}
\label{eq:compactified_objective}
    \frac{x^q}{1+x^q}=z^q.
\end{equation}

The compactified initial set is
\begin{equation}
\label{eq:compactified_initial_set}
    \mathcal Y_0
    :=
    \left\{
        (z,w)\in\mathcal B_q
        : 
        z-x_{\min}w\geq0,
        \quad
        x_{\max}w-z\geq0
    \right\},
\end{equation}
which is precisely the image of \(\mathcal X_0\) under
\eqref{eq:compactification_map}.

Let
\(
    \widetilde f(z,w)
    :=
    w^d f\left(\frac{z}{w}\right)
\)
denote the homogenization of \(f\). The compactified dynamics can then be written as
\begin{align}
\label{eq:compactified_dynamics_homogeneous}
    \dot z
    =
    w^{q+1-d}\widetilde f(z,w), \qquad
    \dot w
    =
    -z^{q-1}w^{2-d}\widetilde f(z,w).
\end{align}

We use so-called occupation measures, which enable us to estimate the performance of an ensemble of trajectories simultaneously and represent the dynamics as a linear constraint in the space of positive Radon measures.

\begin{definition}
\label{def:compactified_measures}
Given \(x_0\in\mathcal X_0\) and
\(
    0\leq\tau\leq T\) with 
    \(\tau\leq t_{\max}(x_0),
\)
let \( \gamma(t) \defeq (z(t),w(t))\) be the compactification
\eqref{eq:compactification_map} of \(x(t;x_0)\), extended continuously to \(t=\tau\) when \(\tau=t_{\max}(x_0)\). The \textit{initial measure} \(\eta_0\), \textit{weighted occupation measure} \(\nu\), and \textit{peak measure} \(\eta_T\) associated with this trajectory and \textit{stopping time} $\tau$ are defined by
\begin{gather*}
    \eta_0
    :=
    \delta_{(0,z(0),w(0))},\qquad \eta_T
    :=
    \delta_{(\tau,z(\tau),w(\tau))},\\
    \nu(\mathcal A)
    :=
    \int_0^\tau w(t)^{-\rho}
        \mathbf 1_{\mathcal A}(t,z(t),w(t))\,dt,\\
\end{gather*}
for every Borel set \(\mathcal A\subset \mathcal K \defeq[0,T]\times\mathcal B_q\), where \(\mathbf{1}_{\mathcal A}\) is the indicator function of \(\mathcal A\). Moreover, a \textit{stopped solution} is a pair $(\tau,\gamma)$.
\end{definition}

For every continuously differentiable function \(\phi\) defined on a neighborhood of \(\mathcal K\), the chain rule along \eqref{eq:compactified_dynamics_homogeneous}, integrated from \(0\) to \(\tau\), gives the following weighted continuity equation
\begin{align}
\label{eq:compactified_liouville}
    \int_{\mathcal K}
    \left(\frac{\partial \phi}{\partial t}w^{\rho}+
         \frac{\partial\phi}{\partial z}F_z +
         \frac{\partial\phi}{\partial w}F_w
    \right)d\nu
    =\int_{\mathcal K}\phi\,d\eta_T - \int_{\mathcal K}\phi\,d\eta_0,
\end{align}
where
\begin{align*}  
    F_z(z,w)
    &:=
    w^{q+1-d+\rho}\widetilde f(z,w), \\
    F_w(z,w)
    &:=
    -z^{q-1}w^{2-d+\rho}\widetilde f(z,w).
\end{align*}
Observe that the weight \(w^{-\rho}\) removes the negative powers of \(w\) from the continuity equation associated with \eqref{eq:compactified_dynamics_homogeneous}, hence, all coefficients in \eqref{eq:compactified_liouville} are polynomial. Note in particular that $F_z, F_w$ are polynomial. This measure reweighting has already been used in \cite{sehnalova2025unbounded}. If \(d\geq3\), then \(\rho=d-2\), and the three coefficients become
\(w^{d-2}\), \(w^{q-1}\widetilde f\), and
\(-z^{q-1}\widetilde f\). If \(d\leq2\), then \(\rho=0\), while
\(q+1-d\) and \(2-d\) are nonnegative.

The mass of every weighted occupation measure generated by an admissible trajectory satisfies
\begin{equation}
\label{eq:uniform_weighted_mass_bound}
    \mass(\nu)\leq C
\end{equation}
because Definition~\ref{def:compactified_measures} and
\eqref{eq:compactification_map} give
\begin{align}
    \mass(\nu)
    &=
    \int_0^\tau w(t)^{-\rho}\,dt
    =
    \int_0^\tau
    \left(1+|x(t)|^q\right)^{\rho/q}dt,
\end{align}
which is bounded by \(C\) according to Lemma~\ref{lem:uniform_integrability}.

We use the same notation $(\eta_0, \nu, \eta_T)$ for arbitrary positive Radon measures satisfying \eqref{eq:compactified_liouville}, with
\begin{gather}\label{eq:support_support_constraints}
        \eta_0\in\mathcal M_+(\{0\}\times\mathcal Y_0),
    \quad
    \nu,\eta_T\in\mathcal M_+(\mathcal K),
    \quad \\
    \mass(\eta_0)=1,
    \quad
    \mass(\nu)\leq C.\notag
\end{gather}

\section{Superposition and no relaxation gap}
\label{sec:no_relaxation_gap}

In this section, we establish a superposition principle tailored to the compactified scalar dynamics, which says that every solution to the weighted continuity equation \eqref{eq:compactified_liouville} corresponds to a superposition of trajectories of the original system.

% SUPERPOSITION---------------------------

We associate with \eqref{eq:compactified_liouville} the artificial-time system
\begin{equation}
\label{eq:compactified_artificial_time_system}
    \frac{dt}{d\theta}=w^\rho,
    \qquad
    \frac{dz}{d\theta}=F_z(z,w),
    \qquad
    \frac{dw}{d\theta}=F_w(z,w),
\end{equation}
where $\theta$ is the artificial time. Notice that \(d\theta=w^{-\rho}dt\) whenever \(w>0\), consistently with the definition of the weighted occupation measure.

\begin{proposition}[Superposition principle]
\label{prop:compactified_superposition}
Let \(\eta_0,\nu,\eta_T\) satisfy the support and mass constraints \eqref{eq:support_support_constraints} and the weighted continuity equation \eqref{eq:compactified_liouville}. Then there exists a probability measure \(\boldsymbol\eta\) on stopped solutions \((\tau,\gamma)\) of the compactified dynamics, where \(0\leq\tau\leq T\) and \(\gamma:[0,\tau]\to\mathcal B_q\) is parametrized by the original time \(t\).

Defining
\[
    e_0(\tau,\gamma):=(0,\gamma(0)),
    \qquad
    e_T(\tau,\gamma):=(\tau,\gamma(\tau)),
\]
we have
\[
    \eta_0=(e_0)_\#\boldsymbol\eta,
    \qquad
    \eta_T=(e_T)_\#\boldsymbol\eta,
\]
and, for every bounded Borel function \(h\),
\begin{equation}
\label{eq:compactified_superposition_formula}
    \int_{\mathcal K} h(t,z,w)\,d\nu
    =
    \int
    \left(
        \int_0^\tau
            h\bigl(t,\gamma(t)\bigr)
            w_\gamma(t)^{-\rho}\,dt
    \right)
    d\boldsymbol\eta(\tau,\gamma),
\end{equation}
where \(w_\gamma\) is the \(w\)-component of \(\gamma\). The represented trajectories are stopped either at a point with \(w>0\) or upon first reaching \(w=0\).
\end{proposition}

\begin{proof}
For \(d\leq2\), the result follows from \cite[Theorem~1]{sehnalova2026infinite}, since \(\rho=0\). We therefore assume \(d\geq3\).

Write \(Y=(t,z,w)\), and let
\[
H(Y):=(w^\rho,F_z(z,w),F_w(z,w)).
\]
For \(Y\in\mathcal K\), denote by \(\Phi_\theta(Y)\) the solution at artificial time \(\theta\) of system \eqref{eq:compactified_artificial_time_system}, with initial condition \(Y\).

Since \(\mathcal B_q\) can be smoothly parametrized by a closed interval and the compactified scalar vector field is nonzero at its endpoints, \(H\) admits a smooth, globally Lipschitz extension whose solutions cannot return to \(\mathcal K\) after leaving it. Indeed, the scalar vector field can be extended beyond the endpoints without changing its sign, while \(dt/d\theta=w^\rho\geq0\) inside \(\mathcal K\).

Every solution starting in \(\mathcal K\) leaves it in finite artificial time. Otherwise, its scalar state coordinate would converge to an equilibrium in \(w>0\), implying that \(dt/d\theta=w^\rho\) is eventually bounded below by a positive constant, contradicting \(t\leq T\).

By compactness and the impossibility of return, there exists \(S<\infty\) such that
\begin{equation}
\label{eq:superposition_uniform_exit}
\Phi_S(\mathcal K)\cap\mathcal K=\varnothing.
\end{equation}
We choose \(S>C\).

\paragraph{Concentration of the measures}
Let
\[
\mathcal K_0:=\operatorname{spt}\eta_0
\subset\{0\}\times\mathcal Y_0
\]
be the support of the initial measure. For \(Y_0\in \mathcal K_0\), define its first exit time from \(\mathcal K\) by
\[
\ell(Y_0)
:=
\inf\{\theta\geq0:\Phi_\theta(Y_0)\notin\mathcal K\}.
\]
From Lemma~\ref{lem:uniform_integrability} and the relation
\(d\theta=w^{-\rho}dt\) we get
\(
\ell(Y_0)\leq C,
\)
including trajectories reaching \(w=0\), for which the bound follows by taking the limit toward the escape time.

Define
\[
Q:=
\{\Phi_\theta(Y_0) :
Y_0\in \mathcal K_0,\ 0\leq\theta\leq S\}, \qquad \mathcal R:=Q\cap\mathcal K.
\]
Thus, \(\mathcal R\) is the set of all points reached by solutions starting in \(\mathcal K_0\), before they leave \(\mathcal K\). We show that both \(\nu\) and \(\eta_T\) are concentrated on this set.

Fix an arbitrary point \(Y_*\in\mathcal K\setminus\mathcal R\). Since \(Q\) and \(\Phi_S(\mathcal K)\) are compact and neither contains \(Y_*\), we can choose a nonnegative smooth function \(h\), positive at \(Y_*\), whose support is disjoint from both sets. Define
\[
\phi(Y):=\int_0^S h(\Phi_\theta(Y))\,d\theta.
\]
By construction, \(\phi\geq0\) and \(\phi=0\) on \(\mathcal K_0\). Moreover, differentiation along the solutions gives
\[
H\cdot\nabla\phi(Y)
=h(\Phi_S(Y))-h(Y)=-h(Y)
\]
for $Y\in\mathcal K$. Substituting $\phi$ into \eqref{eq:compactified_liouville} yields
\[
\int_{\mathcal K}h\,d\nu+
\int_{\mathcal K}\phi\,d\eta_T=0.
\]
Both integrals are nonnegative and therefore vanish. Since \(h(Y_*)>0\) and \(\phi(Y_*)>0\), both measures assign zero mass to some neighborhood of \(Y_*\). Applying this argument to every point outside \(\mathcal R\), and using a countable covering, gives
\begin{equation}
\label{eq:superposition_concentration}
\nu(\mathcal K\setminus\mathcal R)
=
\eta_T(\mathcal K\setminus\mathcal R)=0.
\end{equation}
Hence, all occupation and terminal mass is supported on solutions originating from the prescribed initial set.

\paragraph{Coordinates along the solutions}
We now describe each point reached from the initial set by its initial state and elapsed artificial time. Let
\[
\mathcal K_0^y:=
\{y\in\mathcal Y_0:(0,y)\in \mathcal K_0\}
\]
and define
\[
\mathcal F(y,\theta):=\Phi_\theta(0,y),
\]
which gives the position \((t,z,w)\) of the solution starting from \((0,y)\) after artificial time \(\theta\). We consider this map on the compact set
\[
\mathcal D:=
\{(y,\theta)\in \mathcal K_0^y\times[0,C]:
\mathcal F(y,\theta)\in\mathcal K\}.
\]
By construction, \(\mathcal F(\mathcal D)=\mathcal R\).

The map \(\mathcal F\) is injective on \(\mathcal D\). Indeed, two solutions cannot intersect unless they are the same solution, by uniqueness. Moreover, a solution starting at \(t=0,w>0\) cannot return to \(t=0\), because its physical-time coordinate is strictly increasing before leaving \(\mathcal K\).

The derivative of \(\mathcal F\) is also invertible. At \(\theta=0\), changing \(y\) moves the initial point along \(t=0\), whereas changing \(\theta\) changes its time coordinate at the positive rate \(w^\rho\). These two directions are therefore linearly independent, and they remain so under the smooth flow.

By the inverse function theorem, injectivity and compactness imply that \(\mathcal F\) has a smooth inverse on an open neighborhood \(U\) of \(\mathcal R\) in the extended space-time surface. Indeed, otherwise distinct pairs arbitrarily close to \(\mathcal D\) would have identical images; taking convergent subsequences would contradict either injectivity on \(\mathcal D\) or local invertibility.

Write
\[
\mathcal F^{-1}(Y)=(B(Y),\Theta(Y)).
\]
Here \(B(Y)\) gives the initial state of the solution passing through \(Y\), while \(\Theta(Y)\) gives its elapsed artificial time. By construction,
\begin{equation}
\label{eq:superposition_characteristic_coords}
H\cdot\nabla B=0,
\qquad
H\cdot\nabla\Theta=1,
\qquad
\Theta(0,y)=0.
\end{equation}
The neighborhoods and derivatives above are understood in a smooth two-dimensional extension of \(\mathcal K\); the resulting continuously differentiable test functions admit extensions to a neighborhood in \(\mathbb R^3\), as required by the continuity equation.

\paragraph{Construction of the superposition}
Choose a smooth function \(\chi\), equal to one near \(\mathcal R\) and with compact support in \(U\). For any smooth function \(g(y,\theta)\) defined near \(\mathcal D\), use the test function
\[
\phi(Y)=\chi(Y)g(B(Y),\Theta(Y)),
\]
extended by zero outside \(U\). On \(\mathcal R\), equation~\eqref{eq:superposition_characteristic_coords} gives
\[
H\cdot\nabla\phi
=
\partial_\theta g(B(Y),\Theta(Y)).
\]

Define the measures
\[
\widehat\nu:=(B,\Theta)_\#\nu,
\qquad
\widehat\eta_T:=(B,\Theta)_\#\eta_T,
\]
and let \(\mu_0\) be the projection of \(\eta_0\) onto the initial state space \(\mathcal K_0^y\). Since all three measures are supported on \(\mathcal R\), the continuity equation becomes
\begin{equation}
\label{eq:superposition_coordinate_liouville}
\begin{aligned}
&\int_{\mathcal D}\partial_\theta g(y,\theta)\,
d\widehat\nu = \int_{\mathcal D}g(y,\theta)\,d\widehat\eta_T
-\int_{\mathcal K_0^y}g(y,0)\,d\mu_0.
\end{aligned}
\end{equation}
This transformation simplifies the continuity equation to differentiation with respect to artificial time $\theta$.

First, choose \(g(y,\theta)=u(y)\), where \(u\) is arbitrary and smooth. Equation~\eqref{eq:superposition_coordinate_liouville} then gives
\begin{equation}
\label{eq:superposition_initial_marginal}
(\pi_y)_\#\widehat\eta_T=\mu_0,
\end{equation}
where \(\pi_y(y,\theta)=y\) denotes projection onto the initial state.

Next, for any smooth function \(h\) on \(\mathcal K\), extended smoothly to a neighborhood, choose
\[
g(y,\theta)
=
\int_0^\theta h(\Phi_v(0,y))\,dv.
\]
Since \(g(y,0)=0\) and
\(
\partial_\theta g(y,\theta)
=
h(\Phi_\theta(0,y)),
\)
equation~\eqref{eq:superposition_coordinate_liouville} yields

\begin{equation}
\label{eq:superposition_occupation_identity}
\int_{\mathcal K}h\,d\nu
=
\int_{\mathcal D}
\left(
\int_0^\theta h(\Phi_v(0,y))\,dv
\right)
d\widehat\eta_T(y,\theta).
\end{equation}
Equality of finite Radon measures extends~\eqref{eq:superposition_occupation_identity} to every bounded Borel function \(h\).

For each \((y,\theta)\in \mathcal D\), consider the solution
\[
v\longmapsto\Phi_v(0,y),
\qquad 0\leq v\leq\theta,
\]
stopped at artificial time \(\theta\). Pushing $\widehat\eta_T$ forward under the map assigning to each $(y,\theta)\in\mathcal D$ the corresponding stopped solution gives a measure $\widehat{\boldsymbol\eta}$ on stopped solutions.

Testing \eqref{eq:compactified_liouville} with \(\phi=1\) gives
\(
\operatorname{mass}(\eta_T)
=
\operatorname{mass}(\eta_0)=1,
\)
so \(\widehat{\boldsymbol\eta}\) is a probability measure. Its initial distribution is \(\eta_0\) by~\eqref{eq:superposition_initial_marginal}, its terminal distribution is \(\eta_T\) by construction, and its occupation measure is \(\nu\) by~\eqref{eq:superposition_occupation_identity}. In particular, taking \(h=1\) gives
\(
\operatorname{mass}(\nu)
=
\int_{\mathcal D}\theta\,d\widehat\eta_T(y,\theta)
\leq C,
\)
consistently with the prescribed mass bound.

\paragraph{Return to the original time}
Every represented solution starts with \(w>0\) and remains there until possibly reaching \(w=0\) at its stopping point. Therefore,
\(
\frac{dt}{d\theta}=w^\rho>0
\)
before stopping, so the solution can be reparametrized by the original time \(t\). Denoting the resulting compactified trajectory by \(\gamma(t)\), we have
\[
\dot\gamma(t)
=
w_\gamma(t)^{-\rho}F(\gamma(t)),
\qquad
d\theta=w_\gamma(t)^{-\rho}dt,
\]

where $F \defeq (F_z, F_w)$. The reparametrization is measurable by continuity of the flow and measurability of the inverse time change. Let \(P\) map each stopped artificial-time solution to its reparametrization by the original time. Then \(\boldsymbol\eta:=P_\#\widehat{\boldsymbol\eta}\) is a probability measure on stopped physical-time solutions.

The reparametrization preserves the initial and terminal points, so
\[
\eta_0=(e_0)_\#\boldsymbol\eta,
\qquad
\eta_T=(e_T)_\#\boldsymbol\eta.
\]

Finally, the change of variables in~\eqref{eq:superposition_occupation_identity} gives
\[
\int_{\mathcal K}h\,d\nu
=
\int
\left(
\int_0^\tau
h(t,\gamma(t))w_\gamma(t)^{-\rho}\,dt
\right)
d\boldsymbol\eta(\tau,\gamma),
\]
where the inner integral is understood up to the stopping endpoint when \(w_\gamma(\tau)=0\). This proves the proposition.
\end{proof}

% END SUPERPOSITION------------------------

All coefficients in \eqref{eq:compactified_liouville} are polynomial. We therefore obtain the measure problem with polynomial data
\begin{align}
\label{problem_measure_compact}
    p_{\mathrm c}^*
    :={}&
    \sup_{\eta_0,\nu,\eta_T}
    \int z^q\,d\eta_T
    \\
    \mathrm{s.t.}\quad
    &\eqref{eq:compactified_liouville}
    \text{ holds for all }\phi\in\mathbb R[t,z,w],
    \notag\\
    &\eta_0\in\mathcal M_+(\{0\}\times\mathcal Y_0),
    \qquad
    \mass(\eta_0)=1,
    \notag\\
    &\nu,\eta_T
    \in\mathcal M_+([0,T]\times\mathcal B_q), \qquad \mass(\nu)\leq C.
    \notag
\end{align}
The use of polynomial test functions is justified, because polynomials are dense in continuously differentiable functions on compact sets.

\begin{theorem}
\label{thm:no_relaxation_gap}
There is no relaxation gap between
\eqref{problem_state_bounded_cost} and
\eqref{problem_measure_compact}, that is,
\[
    p_{\mathrm c}^*=p^*.
\]
\end{theorem}

\begin{proof}
Every admissible trajectory in \eqref{problem_state_bounded_cost} generates a feasible triple of measures in \eqref{problem_measure_compact} with the same objective value. Hence \(p_{\mathrm c}^*\geq p^*\).

Conversely, Proposition~\ref{prop:compactified_superposition} decomposes every feasible triple in \eqref{problem_measure_compact} into stopped compactified trajectories. Trajectories terminating at \(w=0\) also have terminal objective values bounded by \(p^*\), since their terminal values are limits of objective values attained at admissible times strictly before finite time escape. The objective of the measure problem is the average of their end point values and is therefore at most \(p^*\). Thus \(p_{\mathrm c}^*\leq p^*\), which proves the claim.
\end{proof}

\section{Convergence of the moment-SOS hierarchy}
\label{sec:convergence}

We finally apply the moment-SOS hierarchy to \eqref{problem_measure_compact}. The supports of all measures are compact basic semialgebraic sets. We include redundant ball inequalities in their descriptions so that the corresponding quadratic modules are Archimedean. For a detailed description of the construction of the moment relaxations see \cite{lasserre_positive_2010}.

Let \(p_k\) denote the optimal value of the moment relaxation of problem \eqref{problem_measure_compact} using test functions of degree \(k\).

\begin{theorem}
\label{thm:moment_convergence}
The moment relaxations provide a monotone sequence of upper bounds satisfying
\begin{equation}
    p_k\downarrow p_{\mathrm c}^*
    \qquad\text{as }k\to\infty.
\end{equation}
\end{theorem}

\begin{proof}
The measure problem \eqref{problem_measure_compact} is a generalized moment problem with polynomial data and compact Archimedean supports. The masses of
\(\eta_0\) and \(\eta_T\) are equal to one, the latter equality following from \eqref{eq:compactified_liouville} tested with \(\phi=1\). The mass of \(\nu\) is uniformly bounded for every $k$ by \(C\). Consequently, every element of the feasible moment sequence is uniformly bounded for each fixed degree. The general convergence result of \cite[Corollary~8]{tacchi2022} therefore applies
and yields \(p_k\downarrow p_{\mathrm c}^*\).
\end{proof}

\section{Numerical example}
Consider a scalar system
\begin{equation}\label{eq:scalar_ex}
    \dot x = x^3 - x.
\end{equation}
Note that this system has an unstable equilibrium point $x = 1$. We estimate the peak over a time horizon $[0, 1]$, hence, $T$ from Section \ref{sec:problem_statement} is set as $T = 1$. For an initial condition $x_0 \in [0.9, 1.2] =: \widehat{\mathcal X}_0$, there are trajectories which blow up in the time horizon $[0, 1]$, whereas for an initial condition $x_0 \in [1.01, 1.02]=: \widetilde{\mathcal X}_0$ every trajectory blows up outside of the time horizon $[0, 1]$. Hence, when we solve problem \eqref{problem_measure_compact} with a compactification of system \eqref{eq:scalar_ex} and initial set $\widehat{\mathcal X}_0$, the compactified peak value is equal to 1. Thus, when we numerically solve the moment relaxations of this problem, their values (which are true upper bounds for the compactified peak value) never drop below 1. In contrast, when we solve \eqref{problem_measure_compact} with a compactification of \eqref{eq:scalar_ex} and an initial set $\widetilde{\mathcal X}_0$, the peak value is finite and the moment relaxations give us a sequence of nonincreasing upper bounds for the true compactified peak value. The moment problem was modeled with GloptiPoly \cite{henrion_gloptipoly_2009} and solved with Mosek \cite{mosek} and the numerical values of the moment relaxations (to 5 significant digits) are shown in Table \ref{tab:numerics}. The constants used in the computations were $q = 2$, $C = 2+ \sqrt{3}$ (this bound was computed from the formula \eqref{eq:explicit_weighted_mass_bound} with $c = 1/2$ and $R = \sqrt{2}$ obtained in Lemma \ref{lem:uniform_integrability}).

\begin{table}[ht]
\centering
\caption{Numerical values of the moment relaxations with test functions of degree $k$ of the peak estimation problem \eqref{problem_measure_compact} of system \eqref{eq:scalar_ex}, where $\hat p_k$ (resp. $\tilde p_k$) denotes the upper bound for peak value with initial set $\widehat{\mathcal X}_0$ (resp. $\widetilde{\mathcal X}_0$).}
\label{tab:numerics}
\begin{tabular}{c|c|c}
     $k$ & $\hat p_k$ & $\tilde p_k$\\
     \hline
     2 & 1.0 & 1.0\\
     4 & 1.0 & 0.97052\\
     6 & 1.0 & 0.58657\\
     8 & 1.0 & 0.58375\\
     10 & 1.0 & 0.58375\\
\end{tabular}
\end{table}

We can see that the compactified peak value upper bounds $\hat p_k$ for the system starting in initial set $\widehat{\mathcal X}_0$ never drop below 1, where 1 represents infinite original peak. This is the desired behaviour, although we emphasize that getting the upper bound 1 in any moment relaxation is not a certificate of blow up. In contrast, the compactified peak value upper bounds $\tilde p_k$ for the system starting in initial set $\widetilde{\mathcal X}_0$ drop below 1. Since the computed values are upper bounds, this is a certificate of a finite peak value. The analytical solution $x(t;x_0)
    = \frac{x_0}
    {\sqrt{x_0^2-(x_0^2-1)e^{2t}}}$ of \eqref{eq:scalar_ex} gives the original peak
\(
    \approx 1.18422,
\)
attained at $x_0=1.02$ and $t=1$. The corresponding compactified peak (obtained by transformation \eqref{eq:weight_exponent}) agrees with the numerical value $0.58375$ obtained from the moment relaxations to 5 significant digits.

 \section{Conclusion}
We have shown that the particular structure of scalar polynomial systems allows us to establish a convergent occupation measure based peak estimation method even in the presence of finite time blow up. The integrability of scalar trajectories ensures the finite mass of the measures required by the compactified formulation. Together with the superposition principle for the weighted continuity equation established in this work, this allows us to prove the absence of a relaxation gap and convergence of the moment-SOS hierarchy.

The main message we want to communicate in this work is the remarkable correspondence between the integrability exponent of scalar trajectories and the exponent required to recover polynomial dynamics after compactification. While our results rely essentially on the scalar structure, investigating whether analogous properties can be identified for broader classes of dynamical systems remains an interesting direction for future research.

\section{Acknowledgments}
 The author would like to thank Didier Henrion, Milan Korda and Martin Kru\v z\'{\i}k for discussions regarding this work. The author acknowledges the use of AI for assistance with brainstorming ideas, mathematical development, coding and drafting the manuscript. The final content, analysis and conclusions remain the sole responsibility of the author.

\bibliographystyle{IEEEtran}
\bibliography{references}

\end{document}